\documentclass{amsart}
\usepackage{bbm}
\usepackage{tikz}

\usepackage{times}
\usepackage{amsfonts}
\usepackage{amsmath}
\usepackage{amscd}
\usepackage{amssymb}
\usepackage{amsxtra}
\usepackage{latexsym}
\usepackage{multirow}
\usepackage{amsthm} 
\usepackage{mathtools}
\usepackage{lmodern}
\usepackage{cjhebrew}
\input{xy}
\xyoption{all}

    \newcommand{\sbt}{\,\begin{picture}(-1,1)(-1,-3)\circle*{2}\end{picture}\ }
\newcommand\sarc{\mathrel{\ooalign{$\nabla$\cr
  \hidewidth\raise.3ex\hbox{$\sbt\mkern5mu$}\cr}}}

\newcommand{\nn}{\,\begin{picture}(-1,1)(-1,-3)\scalebox{.5}{n}\end{picture}\ }
\newcommand\narc{\mathrel{\ooalign{$\nabla$\cr
  \hidewidth\raise.05ex\hbox{$\nn\mkern7mu$}\cr}}}

\theoremstyle{plain}

\newtheorem*{conjectuur*}{Conjecture}

\swapnumbers
\newtheorem{theorem}[subsection]{Theorem}
\newtheorem{corollary}[subsection]{Corollary}
\newtheorem{lemma}[subsection]{Lemma}
\newtheorem{proposition}[subsection]{Proposition}

\theoremstyle{definition}
\newtheorem{definition}[subsection]{Definition}

\newtheorem{example}[subsection]{Example}

\newtheorem{con}[subsection]{Construction}

\theoremstyle{remark}

\newtheorem{remark}[subsection]{Remark}

\swapnumbers

\newcommand{\emptyprop}{q}

\newcommand \ulim[1]{\mbox{ulim}{#1}}
\renewcommand \dim[1]{\mbox{dim}{#1}}
\newcommand \sX[1]{\mathcal{X}_{#1}}
\newcommand \sY[1]{\mathcal{Y}_{#1}}
\newcommand \sS{\mathcal{S}}
\newcommand \sC{\mathcal{C}}
\newcommand \bX{\mathbb{X}}
\newcommand\limsieves[1]{\categ{LSieves}_{#1}}
\newcommand\limssieves[1]{\categ{LsSieves}_{#1}}

\newcommand\smes[1]{\categ{MsSieves}_{#1}}
\newcommand\limssch[1]{\categ{LsSch}_{#1}}
\newcommand\sset{\categ{sSet}}
\newcommand\ssch[1]{\categ{sSch}_{#1}}
\newcommand\sch[1]{\categ{Sch}_{#1}}
\newcommand\set{\categ{Set}}
\newcommand\im[1]{\operatorname{im}({#1})}
\newcommand\ssieves[1]{\categ{sSieve}_{#1}}
\newcommand\mot{\text{\fontsize{16}{26}\selectfont\cjRL{M}}}
\newcommand\lmot{\categ{Lim}\text{\fontsize{16}{26}\selectfont\cjRL{M}}}
\newcommand\mmot{\categ{Mes}\text{\fontsize{16}{26}\selectfont\cjRL{M}}}
\newcommand\bmot{\categ{Base}\text{\fontsize{16}{26}\selectfont\cjRL{M}}}
\newcommand\sO[1]{\mathcal{O}_{#1}}
\newcommand\spec[1]{\operatorname{Spec}(#1)}
\newcommand\sfatpoints[1]{\categ{sFat_{#1}}}

\newcommand \complet[1]{\widehat {#1}}

\renewcommand \hom [3]{\operatorname{Hom}_{#1}(#2,#3)}

\newcommand\into{\hookrightarrow}

\newcommand \maxim{\mathfrak m}
\newcommand \nat{\mathbb N}

\newcommand \op\operatorname

\newcommand{\commdiagram}[9][]{%
\begin{equation}
{\newcommand{\tmpprop}{#1q} 
\if\tmpprop\emptyprop \relax\else \label{#1}\fi}
\begin{aligned}%
\mbox{
\begin{picture}(130,90)%
\put(120,70){\vector( 0,-1){50}}%
\put(10,80){\vector( 1, 0){100}}%
\put(0,70){\vector( 0,-1){50}}%
\put(10,10){\vector( 1, 0){100}}%
\put(115,80){\makebox(0,0)[l]{$#4$}}%
\put(5,80){\makebox(0,0)[r]{$#2$}}%
\put(115,10){\makebox(0,0)[l]{$#9$}}%
\put(5,10){\makebox(0,0)[r]{$#7$}}%
\put(-3,50){\makebox(0,0)[r]{$#5$}}
\put(123,50){\makebox(0,0)[l]{$#6$}}
\put(60,3){\makebox(0,0)[c]{$#8$}}
\put(60,88){\makebox(0,0)[c]{$#3$}}
\end{picture}}
\end{aligned}
\end{equation}}

\newcommand\commtrianglefront[7][]{%
\begin{equation}
{\newcommand{\tmpprop}{#1q} 
\if\tmpprop\emptyprop \relax\else \label{#1}\fi}
\begin{aligned}%
\mbox{
\begin{picture}(120,80)%
\put(55,68){\vector(-1,-2){30}}
\put(65,68){\vector(1,-2){30}}
\put(30,5){\vector(1,0){60}}
\put(60,75){\makebox(0,0)[c]{$#2$}}
\put(25,5){\makebox(0,0)[r]{$#4$}}
\put(95,5){\makebox(0,0)[l]{$#6$}}
\put(60,0){\makebox(0,0)[c]{$#5$}}
\put(37,43){\makebox(0,0)[r]{$#3$}}
\put(83,43){\makebox(0,0)[l]{$#7$}}
\end{picture}}
\end{aligned}
\end{equation}}

\newcommand\commtriangleback[7][]{%
\begin{equation}
{\newcommand{\tmpprop}{#1q}
\if\tmpprop\emptyprop \relax\else \label{#1}\fi}
\begin{aligned}%
\mbox{
\begin{picture}(120,80)%
\put(55,70){\vector(-1,-2){30}}
\put(65,70){\vector(1,-2){30}}
\put(30,5){\vector(1,0){60}}
\put(60,75){\makebox(0,0)[c]{$#2$}}
\put(25,5){\makebox(0,0)[r]{$#6$}}
\put(95,5){\makebox(0,0)[l]{$#4$}}
\put(60,0){\makebox(0,0)[c]{$#5$}}
\put(37,43){\makebox(0,0)[r]{$#7$}}
\put(83,43){\makebox(0,0)[l]{$#3$}}
\end{picture}}
\end{aligned}
\end{equation}}

\usepackage[bookmarks,colorlinks,pagebackref]{hyperref} 

\usepackage{mathabx}
\usepackage{cjhebrew}

\newcommand\fat{\mathfrak z}

\newcommand\categ[1]{\mathbbmss{#1}}

\newcommand\limfat[1]{\complet{\categ{Fat}}_{#1}}
\newcommand\func[1]{#1^\circ}

\newcommand\sieves[1]{\categ{Sieve}_{#1}}

\newcommand\fld{\kappa}

\newcommand \affine[2]{{\mathbb A_{#1}^{#2}}}

\newcommand \arc[2]{\nabla_{\!#1}#2} 
\newcommand \sym[1]{{\langle #1\rangle}}
\newcommand \class[1]{{[ #1]}}

\newcommand \grot[1]{{\mathbf {Gr}(#1)}}
\newcommand \hgrot[1]{{\mathbf {Gr}^h(#1)}}

\newcommand\fatpoints[1]{\categ {Fat}_{#1}}

\newcommand \lef{{\mathbb L}}

\newcommand \mor[3]{\op{Mor}_{#1}(#2,#3)}

\begin{document}
\title{Measurable motivic sites}
\author{Andrew R. Stout}
\address{Andrew R. Stout\\
Graduate Center, City  University of New York\\ 365 Fifth
Avenue\\10016\\U.S.A.   \&  Universit\'{e} Pierre-et-Marie-Curie \\ 4 place Jussieu
\\75005\\Paris, France}
\email{astout@gc.cuny.edu}


\begin{abstract}  
We develop a universal theory of relative simplicial motivic measures.
\end{abstract}

\maketitle
\setcounter{tocdepth}{1}
\tableofcontents

\section{Introduction}

We introduce the notion of a relative limit simplicial partial motivic site and a corresponding notion of motivic measurability which
specializes to the notion of finite schemic motivic measures of H. Schoutens and the notion of infinite schemic motivic measures
of the author. The advantage to working with simplicial motivic sites is that it opens up the possibility of defining the correct notion of the motivic measure of
any derived stack which has some reasonable finiteness conditions. This work was partially supported by the Chateaubriand Fellowship.

\section{Grothendieck rings of simplicial motivic sites}
Let $\Delta$ be the category whose objects are finite, non-empty, totally ordered sets $[n]=\{0,1,\ldots,n\}$ 
and whose morphisms are order preserving functions.
Let $\fatpoints {\fld}$ be the category of connected schemes which are finite over a field $\fld$. We call an element
of $\fatpoints {\fld}$ a fat point over $\fld$. Let
$\Delta^{o}$ be the opposite category of $\Delta$. We denote the category of all covariant 
functors $\Delta^o \to \fatpoints {\fld}$
by $\sfatpoints \fld$, and we call an object in this category a {\it simplicial fat point over} $\fld$. 

We denote by $\ssch \fld$ the category of separated simplicial schemes of finite
type over a field $\fld$. 
We say that a functor $\sX {} : \fatpoints {\fld} \to \sset$ is a {\it
simplicial presieve} and we say that a simplicial presieve
is a $F$-{\it simplicial sieve} if there is a natural incusion of functors
\begin{equation}
\sX {} \into X_{F}^{\circ}:=\hom{\fld}{F(-)}{X} 
\end{equation}
for some covariant functor $F : \fatpoints\fld \to \sfatpoints\fld$ and where $X \in \ssch \fld$. 
We say that $X$ is
an $F$-{\it ambient space} of $\sX{}$.

\begin{remark} \label{rk1}
There are many important functors from $\fatpoints\fld$ to $\sfatpoints\fld$. First, we have the {\it trivial functor} which sends
a fat point $\maxim$ to the functor $[n] \mapsto \maxim$. We denote this functor by $(-)_{\bullet}$. Second, we have the restriction
of the 
{\it fiber functor}, denoted by $(-)_{fib} : \sch \fld \to
\ssch \fld$ defined by sending $Y \in \sch \fld$ to 
the functor $[n] \mapsto Y\times_{\fld} \cdots \times_{\fld} Y$ ($n+1$- fold product). Finally, we have the restiction of the {\it symmetric functor}, 
denoted by $(-)_{sym} : \sch \fld \to
\ssch \fld$, defined by sending $Y \in \sch \fld$ to 
the functor $[n] \mapsto (Y)_{fib}([n])/S_{n+1}$ where $S_{n+1}$ is the permutation group on $n+1$ letters. 

The definitions and results of this paper do not depend on the choice of functor $F$ unless otherwise stated. In the interest of simplicity, 
we will restict our attention to the trivial functor. We will say that a simplicial presieve is a \textit{simplicial sieve}
if it is a $(-)_{\bullet}$-simplicial sieve. We will call a $(-)_{\bullet}$-ambient space of a simplicial sieve an \textit{ambient space}. 
Lastly, we define $X^{\circ}$, $X_{fib}^{\circ}$, and $X_{sym}^{\circ}$ to be
the functors $X_{(-)_{\bullet}}^{\circ}$,  $X_{(-)_{fib}}^{\circ}$, and 
 $X_{(-)_{sym}}^{\circ}$, respectively. 
\end{remark}

Given two simplicial sieves $\sX{}$ and $\sY{}$, we form the simplicial sieves
$\sX{} \times \sY{}$ and $\sX{} \sqcup \sY{}$ as follows.
Each fat point $\maxim \in \fatpoints {\fld}$, determines functors
$\sX{}(\maxim), \sY{}(\maxim) : \Delta^{o} \to \set$. Thus, for each
$\sigma \in \Delta^{o}$, we may define
\begin{equation} \label{firstop}
\begin{split}
(\sX{}(\maxim) \times\sY{}(\maxim))(\sigma):= \sX{}(\maxim)(\sigma) \times
\sY{}(\maxim)(\sigma) \\
(\sX{}(\maxim) \sqcup\sY{}(\maxim))(\sigma):= \sX{}(\maxim)(\sigma) \sqcup
\sY{}(\maxim)(\sigma) \\
\end{split}
\end{equation}
since $\sX{}(\maxim)(\sigma)$ and $\sY{}(\maxim)(\sigma)$ are just sets. The reader may check that these are indeed simplicial sieves.

Given two simplicial sieves $\sX{}$ and $\sY{}$ which share the same ambient
space $X$, 
we form the simplicial sieves $\sX{} \cup \sY{}$, $\sX{} \cap \sY{}$ as
follows. 
Each fat point $\maxim \in \fatpoints {\fld}$, determines functors
$\sX{}(\maxim), \sY{}(\maxim) : \Delta^{o} \to \set$. Thus, for each
$\sigma \in \Delta^{o}$, we may define
\begin{equation} \label{secondop}
\begin{split}
(\sX{}(\maxim) \cup\sY{}(\maxim))(\sigma):= \sX{}(\maxim)(\sigma) \cup
\sY{}(\maxim)(\sigma) \\
(\sX{}(\maxim) \cap\sY{}(\maxim))(\sigma):= \sX{}(\maxim)(\sigma) \cap
\sY{}(\maxim)(\sigma) \\
\end{split}
\end{equation}
since $\sX{}(\maxim)(\sigma)$ and $\sY{}(\maxim)(\sigma)$ are just sets. The
reader may check that these are indeed simplicial sieves.

Given two simplicial sieves $\sX{}$ and $\sY{}$, we say
that a natural transformation $\nu : \sY{} \to \sX{}$ is {\it a morphism of
simplicial sieves} if
given any morphism of simplicial schemes $\varphi : Z \to Y$ such that
$\im{\func{\varphi}} \subset \sY{}$,
there exists a morphism of simplicial schemes $\psi : Z \to X$ with $\sX{}
\subset X$ such that the following diagram commutes 
\[
\xymatrix{
&Z^{\circ}  \ar[d]_{\func{\varphi}}  \ar[rrd]^{\func{\psi}}\\
&\sY{} \ar[r]_{\nu} &\sX{} \ \ar@{^{(}->}[r] &X^{\circ}\\}
\]
This forms the category\footnote{This follows mutatis mutandis from page 12 of
\cite{schmot1}.} of simplicial sieves
which we denote by $\ssieves \fld$. 

\begin{remark} \label{rk2}
 Note the functors $(-)_{\bullet}, (-)_{fib}$, and $(-)_{sym}$ defined in Remark \ref{rk1} extend to functors
 $\sieves\fld \to \ssieves\fld$ in the natural way. 
\end{remark}

\begin{con} \label{admopen}
Given a morphism $s : \sY{} \to \sX{}$ in
$\ssieves \fld$ and a simplicial subsieve $\sX{}' \subset \sX{}$, we define 
{\it the pull-back of} $\sX{}'$ {\it along} $s$, denoted by $s^*\sX{}'$, as the
simplicial subsieve defined by 
\begin{equation}
 s(\maxim)^{-1}\sX{}'(\maxim) : \Delta^o \to \set \ .
\end{equation}
If $\sX{} \in \ssieves \fld$ and $X$ is an ambient space of $\sX{}$, we say that
a subsieve of $\sX{}$ is
an {\it admissible open} of $\sX{}$ if it is of the form $\sX{} \cap U^{\circ}$
where $U$ is an open subset\footnote{Of course, the notion of an open set is relative to choosing grothendieck topology $\tau$ on the category 
$\sch\fld$. For simplicity, we choose $\tau$ to be the zariski topology; however our results hold for general $\tau.$} of $X$. 
Note that this definition does not depend on the ambient space of $\sX{}$. We
say that 
a morphism $s : \sX{} \to \sY{}$ in $\ssieves \fld$ is {\it continuous} if the
pull-back of an admissible open along $s$
is again an admissible open. 

We say that a subcategory $\mot$ of $\ssieves \fld$ is a {\it simplicial motivic
site}
if it is closed under products and if the set
\begin{equation}
\mot|_X := \{ \sX{} \in \mot \mid \sX{}\subset X\}
\end{equation}
forms a lattice with respect to $\cup$ and $\cap$ for each $X\in\ssch \fld$. 
We say that a morphism $s :\sY{} \to \sX{}$ in a simplicial motivic site
$\mot$
is a $\mot$-{\it homeomorphism} if it is continuous and bijective whose inverse
is also continuous. 
\end{con}

\begin{con}\label{firstcon}
This allows us to form the grothendieck ring $\grot{\mot}$ in the following way.
We denote the isomorphism class (defined as a 
$\mot$-homeomorphism) of $\sX{} \in \mot$ by $\sym{\sX{}}$. Then, we denote by
$\grot\mot$ the free abelian group generated by 
$\sym{\sX{}}$ modulo the {\it scissor relations}
\begin{equation}
 \sym{\sX{}\cup \sY{}} + \sym{\sX{}\cap\sY{}}  - \sym{\sX{}} - \sym{\sY{}}
\end{equation}
when $\sX{}$ and $\sY{}$ share the same ambient space $X$. 
We denote by $\class{\sX{}}$ the residue class of $\sym{\sX{}}$ in $\grot\mot$
and for simplicity write $\class X$ for
$\class{X^{\circ}}$. We define multiplication on $\grot\mot$ by
$\class{\sX{}}\cdot\class{\sY{}} := \class{\sX{}\times\sY{}}$.
\end{con}

\begin{remark} 
 Note that the equivalence class of the functor $[n] \to \emptyset$ is $0$ in $\grot\mot$, and the equivalence class of the functor
 $[n] \to \spec \fld$ is $1$ in $\grot\mot$. Therefore, it is natural to define the {\it simplicial leftschetz motive} to be the equivalence
 class of the functor $[n] \mapsto \affine{\fld}{1}$ in $\grot\mot$ if this functor is $\mot$-homeomorphic to an element of $\mot$. We
 will denote the simplicial leftschetz motive by $\lef_{\bullet}$ or sometimes just simply $\lef$ if there is no risk in confusing it
 with its schemic counterpart. 
\end{remark}

\begin{proposition}\label{taka}
Let $F$ be any functor from $ \sieves \fld$ to $\ssieves \fld$ such that the induced map of sets 
  $g_F : \grot{\mot} \to\grot{F(\mot)}$ defined by $g_F(\class{\sX{}}) := \class{F(\sX{})}$ is such that 
  $g_F(0)= 0$ and $g_F(1) = 1$.
 Then, $g_F$ is a ring homomorphism. In particular, this holds for the functors defined in Remark \ref{rk2}, and moreover, in this case,
 $g_F$ is injective.
 \end{proposition}
 
 \begin{proof}
  This follows from Equations \ref{firstop} and \ref{secondop} and by construction of the grothendieck ring.
 \end{proof}

 We write $g_{\bullet}$, $g_{fib}$ and $g_{sym}$ for the injective ring homomorphisms induced by $(-)_{\bullet}$, $(-)_{fib},$ and $(-)_{sym}$, respectively.
 Note that $g_{\bullet}$ and $g_{fib}$ induce injective ring homorphisms from $\grot\mot_{\lef}$ to $\grot{F(\mot)}_{\lef_{\bullet}}$. 

\begin{theorem}\label{simdis}
 Let $\mot$ be a simplicial motivic site.  Then, for all $n \in \nat $, $\tau_n(\mot)$ is a motivic site. For all $n\in \nat$, 
 there is a surjective
 ring homomorphism $h_n$ from $\grot\mot$ to the
 simplicial category of the category corresponding to the grothendieck ring $\grot{\tau_n(\mot)}$.
\end{theorem}

\begin{proof}
 It is clear that $\tau_n(\mot)$ is a motivic site. Temporarly, let $\sC$ denote the linear category with only 
 one object $A$ such that 
 $\mor{\sC}{A}{A}$ is the set of elements of $\grot{\tau_n(\mot)}$ with composition defined via multiplication in $\grot{\tau_n(\mot)}$.
 Let $F$ and $G$ be two covariant functors from $\Delta^o$ to $\sC$. We define the functor $F+G$ by 
 \begin{equation}\label{tplus}
\begin{split}
 (F + G)(\sigma) &= A \quad \forall \sigma \in \Delta^{o} \\
 (F + G)(t) &= F(t) +_{\mor{\sC}{A}{A}} G(t) \quad 
 \forall t \in \mor{\Delta^o}{\sigma}{\sigma'}, \ \forall \sigma, \sigma' \in \Delta^o \ ,\\
 \end{split}
 \end{equation}
and we define the functor $F\cdot G$ by
\begin{equation} \label{ttimes}
\begin{split}
  (F\cdot G)(\sigma) &= A \quad \forall \sigma \in \Delta^{o} \\
 (F\cdot G)(t) &= F(t) \circ_{\mor{\sC}{A}{A}} G(t) \quad 
 \forall t \in \mor{\Delta^o}{\sigma}{\sigma'}, \ \forall \sigma, \sigma' \in \Delta^o\ .\\
 \end{split}
\end{equation}
From this, it is clear how one defines $F - G$ and $F/G$. 
It is easy to check that the covariant functors $\Delta^o \to \sC$ are indeed a
small category whose objects form a commutative ring with a unit under the operations defined above.

Let $t \in \mor{\Delta^o}{\sigma}{\sigma'}$ where $\sigma$ and $\sigma'$ are arbitrary.
We define a $h_n : \grot{\mot} \to \sC$ by sending $\class{\sX{\bullet}}$ to the functor 
$t \mapsto \class{\sX n}$. 
Clearly, $h_n(0) = 0$ and $h_n(1) = 1$. 
Furthermore, when $\sX\bullet $ and
$\sY\bullet $ are disjoint sieves with the same ambient space, $h_n(\class{\sX\bullet} + \class{\sY\bullet})$ is the 
functor $t \mapsto \class{(\sX\ + \sY\ )_n}$ which according to Equation \ref{tplus} is the same as adding the functors
$t \mapsto \class{\sX n}$ and $t \mapsto \class{\sY n}$. Likewise, 
\begin{equation}
 h_n(\class{\sX\bullet}\cdot \class{\sY\bullet}) = h_n(\class{\sX\bullet\times\sY\bullet}) = t \mapsto \class{\sX n \times \sY n}
\end{equation} 
where the far right-hand side is the same as the multiplication of the functors $t \mapsto \class{\sX n}$ and $t \mapsto \class{\sY n}$
according to Equation \ref{ttimes}. Thus, $h_n$ is a ring homomorphism which is clearly surjective for all $n$. 

Let $\sX\bullet \in \mot$ be such that it is not $\mot$-homeomorphic to the empty simplicial sieve. This means that 
for some $n$, $\sX n$ is not the empty sieve. By choosing this $n$, the injectivity of $h_n$ follows. 
\end{proof}

\begin{remark}
Let $\sX\bullet \in \mot$ be such that it is not $\mot$-homeomorphic to the empty simplicial sieve. This implies that 
for some $n$, $\sX n$ is not the empty sieve. Thus, the kernel of $h_n$ is all $\sX{}\in \mot$ such that $\sX n = \emptyset$.
When $h_n$ is injective for some $n$, we say that $\mot$ is {\it strictly schemic}. Clearly, if $\mot = (\mot')_{\bullet}$ where $\mot'$ is some schemic motivic site, then $\mot$ is strictly schemic.
\end{remark}

\begin{remark}
 Let $\mot$ be a motivic site such that the collection of functors $[n] \mapsto \lef^{q(n)}$, where $q: \nat \to \nat$
 is any set map,  whose equivalence classes which lie in $\mot$ are such that they form a multiplicatively stable set $S$. 
 By somewhat abusing notation, we will denote by $\grot{\mot}_{\lef_{\bullet}}$
 the localisation of $\grot\mot$ at the multiplicatively stable set formed from the equivalence classes in $\grot{\mot}$ 
 of elements of $S$.
\end{remark}

\begin{corollary}\label{cordis} Let $\mot$ be a motivic site such that the collection of functors $[n] \mapsto \lef^{q(n)}$, where $q: \nat \to \nat$
 is any set map,  whose equivalence classes which lie in $\mot$ are such that they form a multiplicatively stable set $S$.
 Then for all $n\in\nat$, $h_n(\lef_{\bullet}) = \lef$ 
where $h_n$ is the ring
homomorphism defined in the proof of Theorem \ref{simdis}. Thus, for all $n \in \nat$, $h_n$ induces a surjective ring homomorphism from
$\grot{\mot}_{\lef_{\bullet}}$ to the simplicial categry of the category corresponding to the localised grothendieck ring
$\grot{\tau_n(\mot)}_{\lef}$. 
\end{corollary}

\section{Simplicial adjunction}
For a noetherian scheme $V$, we denote by $\sch{V}$ the category of separated schemes of finite type over $V$. If $X = \spec R$ 
with $R$ a finite $\sO V$-module, then we say that $X$ is a {\it fat} V-{\it point}. We denote the category of fat $V$-points by 
$\fatpoints V$.

Let $V$ and $W$ be two noetherian schemes. We say that there is a {\it schemic adjunction relative to the pair} $(V,W)$ when there 
are functors
\begin{equation}
\begin{split}
 \nu : \fatpoints V \to \fatpoints W \\
 \nabla : \sch W \to \sch V
 \end{split}
\end{equation}
such that there is a natural bijection 
\begin{equation}
 \Theta : \mor W {\nu(\maxim)} X \to \mor V \maxim {\nabla(X)}
\end{equation}
for each $X \in \sch V$.

Let $j : \maxim \to \fld$ be the structure morphism of a fat point $\maxim \in \fatpoints \fld$. 
We have a functor $j^* : \fatpoints \fld \to \fatpoints \maxim$ defined by $j^* \fat : = \fat \times_{\fld} \maxim$
for $\fat \in \fatpoints \fld$ for which there is {\it the schemic agmentation functor} $\nabla_{j^*}$  which is the schemic adjunction to 
$j^*$. We also have a functor $j_* : \fatpoints \maxim \to \fatpoints \fld$ defined by sending $ t : \fat \to \maxim$ to $j\circ t: \fat \to \fld$.
There is also {\it the schemic dimunition functor} $\nabla_{j_*}$ which is the schemic adjunction to $j^*$. The 
details of this can be found in \S 3 of \cite{schmot2}. 
 
 We define {\it the schemic arc functor}
$\nabla_{\maxim}$ 
to be the composition 
\begin{equation}
 \nabla_{\maxim} := \nabla_{j^*} \circ \nabla_{j_*}  . 
\end{equation} 
For each $\maxim \in \fatpoints \fld$, we introduce the $F$-{\it simplicial arc operator} 
\begin{equation}
\sarc_{\maxim}^{F} : \ssieves \fld \to \ssieves \fld
\end{equation}
defined by sending a simplicial sieve $\sX{\bullet}$ to the functor $[n] \to \narc_{\maxim}^{F}\sX\bullet $ where 
\begin{equation}
\narc_{\maxim}^{F}\sX\bullet :=\arc{F(\maxim)_n}\sX n
\end{equation}
and where
$F : \sieves\fld \to \ssieves\fld$ is any covariant functor and where $\nabla_{F(\maxim)_n}$ is the arc operator defined 
in \S 4 of \cite{schmot2} with $F(\maxim)_n := F(\maxim)([n]) \in \fatpoints\fld$. 

\begin{remark}
Note that the face and degeneracy maps of the simplicial sieve $\sarc_{\maxim}^{F}\sX{}$ are induced from applying the arc operator to the face and degeneracy maps of $\sX{}$.
\end{remark}

\begin{remark}
 Following Remarks \ref{rk1} and \ref{rk2}, we denote $\sarc_{\maxim}, \sarc_{\maxim}^{fib}, \sarc_{\maxim}^{sym}$ for the 
 arc operators $\sarc_{\maxim}^{(-)_{\bullet}}, \sarc_{\maxim}^{(-)_{fib}}, \sarc_{\maxim}^{(-)_{sym}}$, respectively. They are called
 the {\it simplical arc operator}, {\it fiber simplicial arc operator}, and the {\it symmetric simplicial arc operator}, respectively. We will
 usually restrict our attention to the simplicial arc operator.
\end{remark}

\section{Limit simplicial sieves}

Previously, there have been two approaches to schemic motivic integration. We have the theory of \textit{finite} schemic motivic integration 
which was developed by Hans Schoutens in \cite{schmot1} and \cite{schmot2}. We call this theory finite because it measures
truncated arc spaces. The author of this paper developed in \cite{me} the beginnings of a theory of \textit{infinite} schemic motivic 
integration which 
measured the volume of arc spaces over certain types of limit points. 
The goal of this section is to generalize the theory of \cite{me} which will inturn subsume the approach of finite schemic motivic integration.
To wit, we must develop the theory of \textit{limit simplicial sieves}.

In \S 7 of \cite{schmot2}, a partial relation on $\fatpoints\fld$ was introduced. This is defined as $\maxim' \leq \maxim $ if $\maxim'$ is a closed subscheme of $\maxim$. We say that a subset $\bX$ of $\fatpoints\fld$ is a {\it point system} and that the direct limit $\fat$ of elements of $\bX$ in the category of locally ringed spaces is a limit point. We denote the full subcategory of locally ringed spaces formed by fat points over $\fld$ by $\limfat\fld$.

This allows us to form the {\it limit (simplicial) arc operator} $\sarc_{\fat}$ relative to a limit point $\fat = \injlim\bX$ to be the projective limit of functors $\sarc_{\maxim}$ where $\maxim \in \bX$.

\begin{definition} \label{deflimss}
Let $\sX \ \in \ssieves {\fld}$.
Let $\fat \in \limfat \fld$ and choose a point system $\bX$ such that $\fat = \injlim \bX$. For each 
$\maxim\in\bX,$ let $\sX {\maxim} \in \ssieves \fld$ be such that there is natural inclusion
\begin{equation}
\sX{\maxim} \into \sarc_{\maxim}{\sX \ }.
\end{equation}
We define $\sX{\star} := \projlim \sX{\maxim}$ and call  it {\it a limit simplicial sieve at the point} $\fat$
{\it with respect to the point system} $\bX$. We call $\sX \ $ {\it the base} of $\sX{\star}$.
\end{definition}

\begin{remark} \label{firstanl}
 Given two limit simplicial sieves $\sX{\star}$ and $\sY{\star}$, we may form 
 $\sX{\star} \times \sY{\star}$ and $\sX{\star} \sqcup \sY{\star}$ in analogue to Equation \ref{firstop}, and 
 we may form $\sX{\star} \cup \sY{\star}$ and $\sX{\star} \cap \sY{\star}$ in analogue to Equation \ref{secondop}.
\end{remark}

\begin{remark} \label{firstanl}
 Given two limit simplicial sieves $\sX{\star}$ and $\sY{\star}$, we may form 
 $\sX{\star} \times \sY{\star}$ and $\sX{\star} \sqcup \sY{\star}$ in analogue to Equation \ref{firstop}, and 
 we may form $\sX{\star} \cup \sY{\star}$ and $\sX{\star} \cap \sY{\star}$ in analogue to Equation \ref{secondop}.
\end{remark}

We form the category of limit sieves, denoted by $\limssieves \fld$, whose objects are functors
\begin{equation}
 \sX{\star} : \fatpoints \fld \to \sset
\end{equation}
where $\sX{\star}$ is a limit sieve at some point $\fat \in \limfat \fld$ with respect to some point 
system $\bX$ and where $\sX \ $ is the 
base of $\sX{\star}$, and whose morphisms are of the form 
\begin{equation}
f_{\star}: \sX{\star}  \to \sY{\star}
\end{equation}
where $f : \sX \ \to \sY \ $ is a morphism of simplicial sieves and where $f_{\star} = \projlim f_{\maxim}$ is defined as usual except that the 
morphisms $f_{\maxim}$ should form a commutative diagram with the inclusions in Definition \ref{deflimss} and the morphism 
$\sarc_{\maxim}{\sX  \ } \to \sarc_{p(\maxim)}{\sY \ }$ induced
by $f$ where $p$ is the morphism of point systems induced by $f_{\star}$. In other words, we should have the following commutative 
diagrams 
\[
\xymatrix{
&\sX\maxim  \ar[d]_{f_{\maxim}}  \ar@{^{(}->}[r] &\sarc_{\maxim}{\sX  \ } \ar[d] \\
&\sY{\maxim} \ar@{^{(}->}[r] &  \sarc_{p(\maxim)}{\sY \ }\\}
\]
in the category $\ssieves \fld$ for all $\maxim \in \bX$.

\begin{example}\label{limschex}
 An important example of a limit simplicial sieve is what we call {\it a limit simplicial scheme}. 
 This occurs when the base $\sX \ $ of a limit simplicial sieve $\sX{\star} \in \limssieves\fld$ 
 is isomorphic to a simplicial scheme $X \in \ssch\fld$ and where
 $\sX{\maxim}$ is isomorphic to $\sarc_{\maxim}X$ for each $\maxim\in \fatpoints\fld$. This implies that
 $\sX{\star}$ is isomorphic to $\sarc_{\fat}X$ in $\limssieves\fld$ where $\fat = \injlim \bX$. 
 We denote
 the full subcategory of $\limssieves \fld$ whose objects are all limit simplicial schemes by $\limssch \fld$. 
 Thus, by Definition \ref{deflimss}, every limit simplicial sieve $\sX{\star}$ is naturally contained in a limit simplicial scheme; 
 we call such a limit simplicial scheme {\it an ambient space} of $\sX{\star}$. 
\end{example}

\begin{example} \label{limscvex}
 Another important example of a limit simplicial sieve is what we call {\it a finite simplicial sieve}. This occurs when 
 the limit point $\maxim$ is in fact an element of $\fatpoints\fld$, or, equivalently, the point system $\bX$ is just a finite set.
 Thus, a limit simplicial sieve in this case is just a simplicial sieve of the form $\sarc_{\maxim} \sX{}$ where $\maxim$
 is any fat point over $\fld$. In particular, by choosing $\maxim = \spec\fld$, we can see that 
 $\ssieves\fld$ is a full subcategory of $\limssieves\fld$ as $\sarc_{\spec\fld}\sX{} = \sX{}$ for any simplicial scheme $\sX{}$. 
\end{example}

\begin{con} We describe the analogue of Construction \ref{admopen} below.
Given a morphism $s : \sY{\star} \to \sX{\star}$ in
$\limssieves \fld$ and another limit simplicial sieve $\sX{\star}' \subset \sX{\star}$, we define 
{\it the pull-back of} $\sX{\star}'$ {\it along} $s$, denoted by $s^*\sX{\star}'$, as the
simplicial subsieve defined by 
\begin{equation}
 s(\maxim)^{-1}\sX{\star}'(\maxim) : \Delta^o \to \set \ .
\end{equation}
If $\sX{\star} \in \limssieves \fld$ and $X$ is an ambient space of $\sX{\star}$, we say that
a subsieve of $\sX{\star}$ is
an {\it admissible open} of $\sX{\star}$ if it is of the form $\sX{\star} \cap U^{\circ}$
where $U$ is an open set of $X$. 
Note that this definition does not depend on the ambient space of $\sX{\star}$. We
say that 
a morphism $s : \sX{\star} \to \sY{\star}$ in $\limssieves \fld$ is {\it continuous} if the
pull-back of an admissible open along $s$
is again an admissible open. 

We say that a subcategory $\mot$ of $\limssieves \fld$ is a {\it limit simplicial motivic
site}
if it is closed under products and if the set
\begin{equation}
\mot|_X := \{ \sX{\star} \in \mot \mid \sX{\star}\subset X\}
\end{equation}
forms a lattice with respect to $\cup$ and $\cap$ for each $X\in\limssch \fld$. 
We say that a morphism  $s :\sY{\star} \to \sX{\star}$ in a simplicial motivic site
$\mot$
is a $\mot$-{\it homeomorphism} if it is continuous and bijective whose inverse
is also continuous. This allows us to form the grothendieck ring $\grot{\mot}$ in analogue to Construction \ref{firstcon}.
\end{con}

\begin{definition} \label{defmes}
We say that a limit simplicial sieve $\sX{\star}  $ is {\it measurable} if there is a non-negative real number $Q$ such that
\begin{equation}\label{defmes}
\mu_{\fat, \bX, Q}^{\sim}(\sX{\star}) :=  \ulim {\class{\sX \maxim}\lef^{-\lceil Q\cdot\dim {\sarc_{\maxim}{\sX \ }}\rceil}}
\end{equation}
is an element of the image of the diagonal ring homomorphism
\begin{equation}
\grot{\ssieves {\fld}}_{\lef_{\bullet}} \to \prod_{\sim} \grot{\ssieves {\fld}}_{\lef_{\bullet}}
\end{equation}
where $\sim$ is some 
ultrafilter\footnote{Clearly, if $\bX$ is finite, then  all ultrafilters are principle. Thus, the ultra-power would be isomorphic to one of its factors. This is exactly what we want as it coincides with the finite motivic measures found in \cite{schmot2}.} on $\bX$ which defines the term on the right-hand side of Equation \ref{defmes} as an ultra-limit. Here
$\lef^{-\lceil Q\cdot\dim {\sarc_{\maxim}{\sX \ }}\rceil}$ is defined to be the multiplicative inverse of the equivalence class of the simplicial scheme defined by
$[n]$ to $\mathbb{A}_{\fld}^{-\lceil Q\cdot\dim {\narc_{\maxim}{\sX \ }}\rceil}$.
We denote the full subcategory of $\limssieves \fld$ whose objects are all measurable limit simplicial sieves as $\smes \fld$. 
\end{definition}

\begin{theorem} The category
 $\smes \fld$ is a partial motivic site which is supported on every measurable limit simplicial scheme. Therefore, we may form 
 its. Therefore, we may form its grothendieck ring $\grot{\smes \fld}$.
\end{theorem}

\begin{proof}
It is obvious that $\smes \fld$ is closed under products. 
 A measurable limit simplicial scheme $X$ is an element of both $\limssch \fld$ and $\smes \fld$. It is clear then that for such
 an $X$, $\smes{\fld}|_X$ is a distributive lattice formed containing $\emptyset$ and $X$ as minima and maxima, respectively. 
\end{proof}

\begin{remark}
 It should be noted that the notation $\limssieves \fld|_{\sX \ }$ and $\smes \fld|_{\sX \ }$ mirrors that of 
 \cite{schmot1}, yet it 
 does not follow that notation exactly. Nothing essential changes by passing to limit sieves vis-a-vis forming the 
 associated Grothendieck rings of limit sieves and measurable sieves.
\end{remark}

\section{Grothendieck ring of a relative limit simplicial motivic site}

\begin{definition} \label{relmot}
Let $\mot_{\fld}$ be a limit simplicial (partial) motivic site. Let $\sS$ be a simplicial sieve. 
 We define $\mot_{\sS}$ to be the full subcategory of 
$\mot_{\fld}$ whose objects $\sX {\star}$ are naturally contained in 
$\sS_{\star} \times (\affine {\fld}{n})^{\circ}$ for some nonnegative integer $n$ and for some limit simplicial sieve $\sS_{\star} \in \mot_{\fld}$
with base $\sS$ such that there is 
a $\mot_{\fld}$-morphism of limit sieves  $j:\sX{\star} \to \sS_{\star}$ which commutes 
with the projection morphism $\sS_{\star} \times (\affine {\fld} {n})^{\circ} \to \sS_{\star}$. We call $j$ the structure morphism.
 Since $\limssieves\fld$ and $\limssch\fld$, we may form the categories 
 $\limssieves\sS$ and $\limssch\sS$, respectively, for any simplicial sieve $\sS$. We call objects in these categories 
 {\it limit simplicial} $\sS$-{\it sieves}, {\it limit simplicial} $\sS$-{\it schemes}, respectively. 
 Likewise, we may form the category $\smes{S}$ of
 {\it measurable limit} $\sS$-{\it sieves}.
\end{definition}

\begin{remark} \label{reductdef}
 Using Example \ref{limscvex}, we can easily define the categories of {\it simplicial} $\sS$-{\it sieves}
 and {\it simplicial} $\sS$-{\it schemes}
 denoted by 
 $\ssieves\sS$ and $\ssch\sS$, respectively, for any simplicial sieve $\sS$. Moreover, the full subcatgories of $\sieves \fld$
 formed by $\tau_0(\ssieves{\sS})$ and $\tau_{0}(\ssch\sS)$ define the notion of $\sS_0$-{\it sieves} and $\sS_0$-{\it schemes}, respectively,
 where $\sS_0$ is any element of $\sS_0 \in \sieves\fld$. 
 \end{remark}

\begin{definition} \label{goob}
Given a morphism $s : \sY{\star} \to \sX{\star}$ in
$\limssieves \sS$ and a simplicial $\sS$-subsieve $\sX{\star}' \subset \sX{\star}$, we define 
{\it the pull-back of} $\sX{\star}'$ {\it along} $\nu$, denoted by $s^*\sX{\star}'$ as the
simplicial $\sS$-subsieve defined by 
\begin{equation}
 s(\maxim)^{-1}\sX{\star}'(\maxim) : \Delta^o \to \set \ .
\end{equation}
If $\sX{\star} \in \limssieves \sS$ and $X$ is an ambient space of $\sX{\star}$, we say that
a subsieve of $\sX{\star}$ is
an $\sS$-{\it admissible open} of $\sX{\star}$ if it is of the form $\sX{\star} \cap U^{\circ}$
where $U$ is an open set of $X$. 
Note that this definition does not depend on the ambient space of $\sX{\star}$. We
say that 
a morphism $s : \sX{\star} \to \sY{\star}$ in $\limssieves \sS$ is $\sS$-{\it continuous} if the
pull-back of an $\sS$-admissible open along $s$
is again a $\sS$-admissible open. 
We say that a subcategory of $\mot$ of $\limssieves\sS$ is a {\it limit simplicial} $\sS$-{\it motivic site} if it is closed under products
and for any $X \in \limssch\sS$, $\mot|_X$ is a complete latice with respect to $\cup$ and $\cap$. 
We say that a morphism $s :\sY{\star} \to \sX{\star}$ in a simplicial motivic site
$\mot$
is a $\mot$-{\it homeomorphism} if it is continuous and bijective whose inverse
is also continuous.
\end{definition}

\begin{remark}
 As in Remark \ref{reductdef}, we may use Definition \ref{goob} and Example \ref{limscvex}, to define the notion of 
 a {\it simplicial} $\sS$-motivic site for any simplicial sieve $\sS$. We may use this definition together with the map $\tau_0$, to define the notion 
 of a $\sS_0$-motivic site for any $\sS_{0} \in \sieves\fld$.
\end{remark}

\begin{theorem}
 Let $\mot_{\fld}$ be a limit simplicial (partial) motivic site and let $S$ be a simplicial scheme. Then, $\mot_{S}$ is also a limit simplicial (partial) 
 $S$-motivic site as defined in Definition 
 \ref{relmot}. Moreover, when $\sS$ is a simplicial sieve, $\mot_{\sS}$ is a (partial) $\sS$-motivic site supported on all limit simplicial 
 $\sS$-schemes which are also elements of $\mot_{\sS}$.
 Futhermore, this is enough to define its grothendieck ring $\grot{\mot_{S}}$ (respectively, $\grot{\mot_{\sS}}$).
\end{theorem}

\begin{proof}
Let $\sS$ be any simplicial sieve.
 We first show that $\mot_{\sS}$ is closed under products. Thus, let $\sX{\star}$ and $\sY{\star}$ be two elements of $\mot_{\sS}$ which are 
 contained in two limit simplicial sieves $\sS_{\star}$ and $\sS_{\star}^{\prime}$, respectively, with base $\sS$. 
 In particular, this means that there are limit points $\fat = \injlim \bX$ and $\fat^{\prime} = \projlim \bX^{\prime}$ 
 such that $\sS_{\star} = \projlim_{\maxim\in\bX} \sS_{\maxim}$ and $\sS_{\star}^{\prime} = \projlim_{\maxim^{\prime}\in\bX^{\prime}}\sS_{\maxim}^{\prime}$
 where $\sS_{\maxim}$ and $\sS_{\maxim^{\prime}}$ are naturally contained in $\sarc_{\maxim} \sS$ and $\sarc_{\maxim^{\prime}}\sS$, respectively.
 Thus, we set $\bX^{\prime\prime}$ to be the point system formed by $\maxim \times_{\fld} \maxim^{\prime}$ with 
 $\maxim \in \bX$ and $\maxim^{\prime} \in \bX^{\prime}$ so that 
 \begin{equation}
 \sS_{\maxim} \times \sS_{\maxim^{\prime}}^{\prime} \subset \sarc_{\maxim \times \maxim^{\prime}} \sS .
 \end{equation}
 Then it is immediate that the $\sS_{\star} \times \sS_{\star}^{\prime} = \projlim\sS_{\maxim} \times \sS_{\maxim^{\prime}}^{\prime}$
 is a limit simplicial sieve with base $\sS$ such that  $\sX{\star}\times \sY{\star}$ is naturally contained in $\sS_{\star} \times \sS_{\star}^{\prime}\times(\affine{\fld}{n+m})^{\circ}$.
 Since $\mot_{\fld}$ is a limit simplicial site, we have $\sX{\star}\times \sY{\star} \in \mot_{\fld}$. Furthermore,
 it is clear that we can define a $\mot_{\fld}$-morphism from $\sX{\star}\times \sY{\star}$ to $\sS_{\star} \times \sS_{\star}^{\prime}$
 by taking the product of the morphisms $\sX{\star} \to \sS_{\star}$ and $\sY{\star} \to \sS_{\star}^{\prime}$. 
 Since the projection of the products is the product of the projections, this morphism will commute with the projection
 $\sS_{\star} \times \sS_{\star}^{\prime}\times(\affine{\fld}{n+m})^{\circ} \to \sS_{\star} \times \sS_{\star}^{\prime}$. Thus,
 $\mot_{\sS}$ is closed under products.
 
 Now, let $X \in \limssch\fld$ and let $S$ be a simplicial scheme. We need to show that $\mot_{S}|_{X}$ is a complete lattice. The fact that $\emptyset$ and $X$
 are in $\mot_{S}|_X$ and the respective minimum and maximum is obvious. Furthermore, it is clear that the union and intersection of
 two elements of $\mot_{S}|_{X}$ is again an element of $\mot_{S}|_{X}$. 
 
 This is enough to define the grothendieck group as nothing essential changes. For good measure, we repeat the construction here. 
We denote the isomorphism class (defined as a 
$\mot_{S}$-homeomorphism) of $\sX{\star} \in \mot_S$ by $\sym{\sX{\star}}$. Then, we denote by
$\grot{\mot_{S}}$ the free abelian group generated by 
$\sym{\sX{\star}}$ modulo the {\it scissor relations}
\begin{equation}
 \sym{\sX{\star}\cup \sY{\star}} + \sym{\sX{\star}\cap\sY{\star}}  - \sym{\sX{\star}} - \sym{\sY{\star}}
\end{equation}
when $\sX{\star}$ and $\sY{\star}$ share the same ambient space $X$. 
We denote by $\class{\sX{\star}}$ the residue class of $\sym{\sX{\star}}$ in $\grot{\mot_{S}}$.
We define multiplication in $\grot{\mot_{\sS}}$ by
$\class{\sX{\star}}\cdot\class{\sY{\star}} := \class{\sX{\star}\times_{\sS}\sY{\star}}$. Note that this makes $\class\sS$ the multiplicative
unit. 
\end{proof}

\begin{remark}
The fiber product of sieves $\sX{} \times_{\sS} \sY{}$ is computed pointwise. There is no problem here as the fiber product exists in 
the category of sets. We define multiplication this way with the expressed purpose to make $\sS$ the multiplicative unit. This way
we have a good notion of a relative leftschetz motive.
 More clearly, we denote by $(\lef_{\sS})_{\bullet}$,
$\lef_{\bullet}$, $\lef_{\sS}$, or just $\lef$ for the equivalence class of  $\sS \times \affine{\fld}{1}$. 
\end{remark}

\begin{corollary}
 Let $S$ be a simplicial scheme. Then $\limssieves{S}$ and $\limssch{S}$ are limit simplicial $S$-motivic sites. Thus, we may form
 their grothendieck rings $\grot{\limssieves{S}}$ and $\grot{\limssch{S}}$, respectively. Also, $\smes{S}$ is a limit simplicial
 partial motivic site supported on all measurable simplicial $S$-schemes
 which means that we may form its grothendieck ring $\grot{\smes{S}}$. Moreover, given any
 simplicial sieve $\sS$, $\limssieves{\sS}$, $\limssch{\sS}$, and $\smes{\sS}$ are partial motivic sites supported
 on all limit simplicial $\sS$-schemes (respectively, all measurable simplicial $\sS$-schemes). Thus, we may 
 form their grothendieck rings $\grot{\limssieves{\sS}}$ $\grot{\limssch{\sS}}$, and $\grot{\smes{\sS}}$. 
\end{corollary}

\begin{remark} 
Let $S$ be a simplicial scheme
 Thus, the sets $\ssieves{S}$ and $\ssch{S}$ are simplicial $S$-motivic sites, and the sets $\sieves{S}$, and $\sch{S}$ are $S$-motivic sites.
 Therefore, we may form their grothendieck rings $\grot{\ssieves{S}}$, $\grot{\ssch{S}}$, $\grot{\sieves{S}}$, and $\grot{\sch{S}}$. 
Moreover, if $\sS$ is a simplicial sieve, $\ssieves{\sS}$ and $\ssch{\sS}$ 
are partial motivic sites. Thus, we may also form their grothendieck rings: $\grot{\ssieves{\sS}}$ and $\grot{\ssch{\sS}}$. 
 \end{remark}

\begin{theorem}\label{simreldis}
 Let $\mot$ be a simplicial $\sS$-motivic site where $\sS$ is any 
 simplicial scheme.  Then, for all $n \in \nat $, $\tau_n(\mot)_{\tau_n(\sS)}$ is a $\tau_n(\sS)$-motivic site. For all $n\in \nat$, 
 there is a surjective
 ring homomorphism $h_n$ from $\grot{\mot}$ to the
 simplicial category of the category corresponding to the grothendieck ring $\grot{\tau_n(\mot)}$. 
\end{theorem}

\begin{proof}
 It is clear that $\tau_n(\mot)$ is a $\tau_n(\sS)$-motivic site. The proof is exactly the same as the proof of Theorem \ref{simdis}. 
\end{proof}
\begin{remark}
Let $\sS$ be a simplicial sieve.
 Let $\mot$ be a simplicial $\sS$-motivic site such that the collection of functors $[n] \mapsto \sS\times\lef^{q(n)}$, where $q: \nat \to \nat$
 is any set map,  whose equivalence classes which lie in $\mot$ are such that they form a multiplicatively stable set $S$. 
 By somewhat abusing notation, we denote by $\grot{\mot}_{\lef_{\bullet}}$
 the localisation of $\grot\mot$ at the multiplicatively stable set formed of equivalence classes in $\grot{\mot}$ 
 of elements of $S$.
\end{remark}

\begin{corollary}\label{correldis} Let $\mot$ be a simplicial motivic site where
$\sS$ is any simplicial scieve such that the collection of functors $[n] \mapsto \sS\times\lef^{q(n)}$, where $q: \nat \to \nat$
 is any set map,  whose equivalence classes which lie in $\mot$ are such that they form a multiplicatively stable set $S$. 
 Then for all $n\in\nat$, $h_n(\lef_{\bullet}) = \lef$ 
where $h_n$ is the ring
homomorphism defined in the proof of Theorem \ref{simreldis}.
Thus, for all $n\in \nat$, $h_n$ induces a surjective ring homomorphism from 
 $\grot{\mot}_{\lef_{\bullet}}$ to the simplicial category of the category corresponding to $\grot{\tau_n(\mot)}_{\lef}$. 

\end{corollary}

\begin{definition} \label{defrelmes}
Let $S$ be a simplicial scheme.
We say that a limit simplicial $S$-sieve $\sX{\star}$ is $S$-{\it measurable} if there is a non-negative real number $Q$ such that
\begin{equation}\label{defrelmes}
\mu_{\fat, \bX, Q}^{\sim}(\sX{\star}) :=  \ulim {\class{\sX \maxim}\lef^{-\lceil Q\cdot\dim {\sarc_{\maxim}{\sX \ }}\rceil}}
\end{equation}
is an element of the image of the diagonal ring homomorphism
\begin{equation}
\grot{\ssieves {S}}_{\lef_{\bullet}} \to \prod_{\sim} \grot{\ssieves {S}}_{\lef_{\bullet}}
\end{equation}
where $\sim$ is the 
ultrafilter on $\bX$ which defines the term on the right-hand side of Equation \ref{defrelmes} as an ultra-limit and $\sX{}$ is the base
of $\sX{\star}$. Here
$\lef^{-\lceil Q\cdot\dim {\sarc_{\maxim}{\sX \ }}\rceil}$ is defined to be the multiplicative inverse of the equivalence class of the simplicial scheme defined by
$[n]$ to $\mathbb{A}_{S}^{-\lceil Q\cdot\dim {\narc_{\maxim}{\sX \ }}\rceil}$. We denote
this element by $\lef^{-\lceil Q\cdot\dim {\sarc_{\maxim}{\sX \ }}\rceil}$.
We temporarily denote the full subcategory of $\limssieves S$ whose objects are all measurable limit simplicial $S$-sieves as $\overline{\smes S}$. 
\end{definition}

\begin{proposition}
 The Definition \ref{relmot} and Definition \ref{defrelmes} are indeed the same. In other words, there is an equivalence of categories
 between $\smes S$ and $\overline{\smes S}$. 
\end{proposition}

\begin{proof} Let $S$ be any simplicial scheme. 
Note that $\overline{\smes S}$ includes into $\smes \fld$ and $\limssieves S$ by definition. Thus, $\overline{\smes S}$ includes into
$\smes S$. Thus, we are left with showing that if $\sX{\star}$ is an element of $\smes S$, then it satisfies Definition \ref{defrelmes}.
This follows from the fact that there are injective ring homomorphisms 
\begin{equation}
\begin{split}
 \grot{\ssieves{S}} \hookrightarrow \grot{\ssieves\fld} \\
 \grot{\sieves{\tau_0(S)}} \hookrightarrow \grot{\sieves\fld} \ . \\
 \end{split}
\end{equation}
\end{proof}

\begin{definition}
 Let $\mot$ be a simplicial (partial) $\sS$-motivic site where $\sS$ is a simplicial sieve. We denote by $\lmot$ the full
 subcategory of $\limssieves \sS$ formed of all
 limit simplicial $\sS$-sieves whose base is in $\mot$. We denote by $\mmot$ the full
 subcategory of $\smes \sS$ formed of all
 limit simplicial $\sS$-sieves whose base is in $\mot$.  In particular, we have natural containments
 \begin{equation}
  \mot \subset \mmot \subset \lmot \ .
 \end{equation}

\end{definition}

\begin{example}
 As a trivial example, it is easy to see that the category $\categ{LimsSieves}_{\sS}$ is naturally isomorphic to $\limsieves{\sS}$
 and $\categ{MessSieves}_{\sS}$ is naturally isomorphic to $\smes{\sS}$ for
 any simplicial sieve $\sS$. However, it should be noted that $\categ{LimsSch}_{\sS}$ properly contains $\limssch \sS$. 
\end{example}

\begin{proposition} If $\mot$ is  simplicial (partial) $\sS$-motivic site where $\sS$ is a simplicial sieve. Then,
 $\lmot$ and $\mmot$ are limit simplicial (partial) $\sS$-motivic sites. 
\end{proposition}

\begin{proof}
 This follows directly from the definitions.
\end{proof}

Let $\mot$ be a limit simplicial (partial) $\sS$-motivic site for some simplicial scheme $\sS$. There is a functor from $\mot$ to 
$\ssieves \sS$ which associates to each object $\sX{\star}$ in $\mot$ the base $\sX{}$ of $\sX{\star}$. We denote the image of this functor by $\bmot$.

\begin{proposition} If $\mot$ is  limit simplicial (partial) $\sS$-motivic site where $\sS$ is a simplicial sieve. Then,
 $\bmot$ is a  simplicial (partial) $\sS$-motivic sites. 
\end{proposition}

\begin{proposition}
 Let $\sS$ be a simplicial scheme. Let $\mot$ be a simplicial (partial) $\sS$-motivic site. Then, the limit simplicial (partial)
 $\sS$-motivic site $\categ{BaseLim}\mot$ is natural isomorphic to $\mot$. Moreover, if $\mot$ is a limit simplicial (partial) $\sS$-motivic sites,
 then there is a natural inclusion
 \begin{equation}
  \mot \subset \categ{LimBase}\mot \ . 
 \end{equation}
\end{proposition}

\begin{definition}
 Let $\sS$ be a simplicial scheme and let $\mot$ be a limit simplicial (partial) $\sS$-motivic site. We call the 
 the limit simplicial (partial) $\sS$-motivic site
 $\categ{LimBase}\mot$ the {\it closure} of $\mot$ and denote it by $\bar\mot$. We say that $\mot$ is {\it closed} if $\mot = \bar\mot$. 
\end{definition}

Let $f: \sS \to \sS^{\prime}$ be a morphism in $\ssieves\fld$. We then have induced functors
\begin{equation}
 \begin{split}
  f_{*} : \limssieves{\sS} \to \limssieves{\sS^{\prime}}&, \ [\sX{\star} \to \sS_{\star}] \mapsto [\sX{\star}\to (\sS^{\prime})_{\star}] \\
  f^{*} : \limssieves{\sS^{\prime}} \to \limssieves{\sS}&, \ [\sX{\star} \to (\sS^{\prime})_{\star}] \mapsto [\sX{\star}\times_{\sS^{\prime}}\sarc_{\fat}\sS \to \sarc_{\fat}\sS]  
 \end{split}
\end{equation}

\begin{proposition} \label{j1}
 $f_{*}$ is the left-adjoint to $f^{*} .$ 
\end{proposition}

\begin{proof}
 Let $\sX{\star} \in \limssieves{\sS}$ and let $\sY{\star} \in \limssieves{\sS^{\prime}}$. We send a morphism
 $t : f_*(\sX{\star}) \to \sY{\star}$ to the morphism obtained by applying the functor $f^*$. Thus, 
 \begin{equation}
  [t : f_*(\sX{\star}) \to \sY{\star}] \mapsto  [f^*t : \sX{\star}\times_{\sS^{\prime}}\sarc_{\fat}\sS \to \sY{\star}\times_{\sS^{\prime}}\sarc_{\fat}\sS]
 \end{equation}
By definition of fiber product, this induces a morphism from $\sX{\star}$ to $f^*(\sY{\star})$. 
Now, we construct an inverse to this set map. We send a morphism $s : \sX{\star} \to f^*(\sY{\star})$ to 
the morphism obtained by applying the functor $f_*$. Thus, $s$ is mapped to the $(id\times f) \circ s$ which is a morphism
from $\sX{\star}$ to $\sY{\star} \times_{\sS^{\prime}} \sarc_{\fat}\sS^{\prime} \cong \sY{\star}$.
It is a straightforward exercise to show that the bijection here is natural. 
\end{proof}

\begin{definition}
 Let $\sS$ and $\sS^{\prime}$ be simplicial sieves. Let $\mot_{\sS}$ (resp., $\mot_{\sS'}$) be a limit simplicial (partial) 
 $\sS$-motivic site (resp., $\sS^{\prime}$-motivic site). We say that a functor 
 \begin{equation}
  F: \mot_{\sS} \to \mot_{\sS^{\prime}}
 \end{equation}
is a {\it morphism of (partial) motivic sites} if $F(\sS{}) = \sS{}'$, $F(\emptyset)=\emptyset$, preserves products, and for all limit simplicial
$\sS$-schemes $X$, the functor 
\begin{equation}
 F|_X : \mot_{\sS}|_X \to \mot_{\sS^{\prime}}|_{\overline{F(X)}}
\end{equation}
is a morphism of lattices where $\overline{F(X)}$ is the zariski closure\footnote{This is defined as the intersection of all limit simplicial 
schemes containing $F(X)$.} of $F(X)$. 
A bijective morphism of (partial) motivic sites whose set-theoretic inverse is
also a morphism of (partial) motivic sites which also defines an equivalence of categories is called an {\it isomorphism
of motivic sites.} 
\end{definition}

\begin{theorem} \label{j2}
 Let $\sS$ and $\sS^{\prime}$ be simplicial sieves. Let $\mot_{\sS}$ (resp., $\mot_{\sS'}$) be a limit simplicial (partial) 
 $\sS$-motivic site (resp., $\sS^{\prime}$-motivic site). 
 Let $F : \mot_{\sS} \to \mot_{\sS^{\prime}}$ be a morphism of motivic sites. This induces a ring homomorphism
 \begin{equation}
  F : \grot{\mot_{\sS}} \to \grot{\mot_{\sS^{\prime}}} \ .
 \end{equation}
If $F$ is an surjective (resp., injective), then so is the induced ring homomorphism. Moreover, 
if $F$ is an isomorphism of motivic sites, then then the induced ring homomorphism is an isomorphism of rings.
\end{theorem}

\begin{proof} 
 This is immediate.
\end{proof}

\begin{corollary}\label{j3}
 Let  $f: \sS \to \sS^{\prime}$ be a morphism in $\ssieves\fld$. 
 Let $\mot_{\sS}$ (resp., $\mot_{\sS^{\prime}}$) be a limit simplicial (partial) $\sS$-motivic site (resp., $\sS^{\prime}$-motivic site). 
 Then, $f_*$ and $f^*$ are morphisms of motivic sites between $\mot_{\sS}$ and $\mot_{\sS^{\prime}}$. 
 Thus, we have induced ring homomorphisms
 \begin{equation}
  \begin{split}
   f_{*} : \grot{\mot_{\sS}} &\to \grot{\mot_{\sS^{\prime}}} \\
   f^{*} : \grot{\mot_{\sS^{\prime}}} &\to \grot{\mot_{\sS}} \ .
  \end{split}
 \end{equation}
\end{corollary}

\subsection{Alternative construction}

We present here a somewhat finer concept of relative motivic sites. As of yet, we have defined a motivic site relative to some simplicial sieve 
$\sS{}$. By definition, all limit simplicial sieves with base $\sS{}$ are naturally contained in $\sarc_{\fat}\sS{}$ for some limit point $\fat$. Thus, 
every element $\sX{\star}$ of $\mot_{\sS{}}$ is contained in $\sarc_{\fat}\sS{} \times \mathbb{A}_{\fld}^{n}$ for some $n$ and comes equipped 
with a structure morphism $j : \sX{\star}\to \sarc_{\fat}\sS{}$ satisfying the properties of a limit simplicial sieve -- i.e., it is a projective limit of 
morphisms
$\sX{\maxim} \into \sarc_{\maxim}\sS{}$ where $\maxim$ run over $\bX$. Given a limit simplicial sieve $\sS_{\star}$ with base $\sS{}$, we may 
form the subset $\mot_{\sS_{\star}}$ of $\mot_{\sS{}}$ formed from all elements $\sX{\star}$ which are naturally contained in 
$\sS_{\star} \times \mathbb{A}_{\fld}^{n}$. This subset is indeed a partial motivic site if we define multiplication by $\sX{\star}\times_{\sS{\star}}\mathcal{Y}_{\star}.$  Note that $\sS_{\star}$  is the multiplicative unit in $\mot_{\sS_{\star}}$. The results of this 
section carry over to these types of motivic sites.  In particular, Proposition \ref{j1}, Theorem \ref{j2}, and Corollary \ref{j3} have corresponding analogous results.

\section{Action of $\tau_0$}
Let $\sS \in \ssieves\fld$. 
There is a functor $\tau_n : \ssieves \sS \to \sieves \sS$ defined by sending a
simplicial $\sS$-sieve $\sX{\bullet}$ to the $\sS$-sieve 
$\sX{n}$ and a morphism of simplicial $\sS$-sieves $f_{\bullet} : \sX{\bullet} \to
\sY{\bullet}$ to the morphism of $\sS$-sieves 
$f_n : \sX{n} \to \sY{n}$. 

\begin{lemma}\label{hack}
The functor $\tau_n :\ssieves \sS \to \sieves\sS $ is the 
left adjoint to the trivial functor $(-)_{\bullet} : \sieves \sS  \to \ssieves\sS$
\end{lemma}
\begin{proof}
 Let $\sX{\bullet} \in \ssieves \sS$ and let $\sY{} \in \sieves\sS$. Consider a morphism 
 simplical $\sS$-sieves $t_{\bullet}: (\sY{})_{\bullet} \to \sX{\bullet}$. We send it to the morphism of $\sS$-sieves defined by 
 $\tau_{n}(t_{\bullet}) = t_n : (\sY{})_{n} \to \sX n$. Note that $(\sY{})_n = \sY{}$. Thus, we have a surjective map of sets
 \begin{equation}
  \mor{\ssieves\sS}{(\sY{})_{\bullet}}{\sX\bullet} \to \mor{\sieves\sS}{\sY{}}{\tau_n(\sX{\bullet})} \ .
 \end{equation}
Likewise, we define an inverse to this map by sending a morphism of $\sS$-sieves $s : \sY{} \to \sX n$ to the morphism of simplicial
$\sS$-sieves defined by $(s)_{n} : (\sY{})_{\bullet} \to (\sX n)_{\bullet}$. It is immediate that this is the inverse. 
The fact that these bijections are natural is a simple exercise.
\end{proof}

\begin{corollary}\label{jack}
 The functors $\tau_n : \ssieves \sS \to \sieves \sS$ are naturally isomorphic. 
\end{corollary}

\begin{proof}
 This is by uniqueness of left adjoints. 
\end{proof}

\begin{remark}
 As a result of Corollary \ref{jack}, we may often work just with the functor $\tau_0$. When we restrict to a (partial) $\sS$-motivic site,
 sometimes this result will not hold. 
\end{remark}

\begin{proposition}
 The functor $\tau_n : \ssieves \sS \to \sieves \sS$ induces a surjective ring
homomorphism
 \begin{equation}
  \grot{\mot} \to \grot{\tau_n(\mot)}
 \end{equation}
where $\mot$ is a limit simplicial (partial) $\sS$-motivic site and where $\tau_n(\mot)$ is the limit (partial) $\tau_n(\sS)$-motivic
site defined by the image of $\mot$ under $\tau_n$. If $\mot$ is strictly schemic; then this is an isomorphism of rings with inverse induced from the trivial functor $(-)_{bullet}$ if both of these functors are full and faithful.
\end{proposition}

\begin{proof}
This is immediate.
\end{proof}

\begin{corollary}
 Let $\mot$ be a limit (partial) $\sS_{0}$-motivic site for some sieve $\sS_{0}$. Then $(\mot)_{\bullet}$ is a limit simplicial (partial) $(\sS_0)_{\bullet}$-motivic site.
 Futhermore, there is an isomorphism
 \begin{equation}
  \grot{(\mot)_{\bullet}} \to \grot{\mot} 
 \end{equation}
Note that in this case $\tau_n$ are all naturally isomorphic.
\end{corollary}

\begin{remark}
 The obstruction to $\grot{\mot} \to \grot{\tau_n(\mot)}$ being an isomorphism is that in general every morphism of simplicial sieves
 will not arrise from applying $(-)_{\bullet}$ to a morphism of sieves and that there can be more than
 one morphism of simplicial sieves which are sent to the same morphism of sieves under $\tau_n$. 
\end{remark}

Let $\sS$ be a simplicial scheme. We say that an element of $\sieves {\tau_0(\sS)}$ is {\it measurable} if it is an element of 
$\tau_0(\smes{\sS})$. As $\tau_0(\ssieves{\sS}) = \sieves{\tau_0(\sS)}$, this is enough to define the notion for all relative limit sieves.

 Let $\sS$ be a simplicial scheme. Let $\mot_{\sS}$ be a limit simplicial (partial) $\sS$-motivic site. For each $n$,
 let $\mot_{\sS_n} = \tau_n(\mot_{\sS})$. Let $\categ{G}_{\bullet}$ be the simplicial category defined by $[n] \mapsto \mot_{\sS_n}$ with
 face maps $(d_n)_{*}$ and degeneracy maps $(s_n)_{*}$. Also, let $\categ{H}_{\bullet}$ be the cosimplicial category defined by $[n] \mapsto \mot_{\sS_n}$ with
 face maps $(d_n)^{*}$ and degeneracy maps $(s_n)^{*}$. Then, by the simplicial yoneda embedding, the category formed by the collection of functors  from $\Delta^{\circ}$ to $\sieves {\sS}$ defined by  $\mot_{\sS}$ is naturally equivalent to 
 $\categ{G}_{\bullet}$ and naturally anti-equivalent to $\categ{H}_{\bullet}$.
Thus, there are two possible ways to endow $\mot_{\sS}$ with the structure of a simplicial category. 

Likewise, let $G_{\bullet}$ be the simplicial ring defined by $[n] \mapsto \grot{\mot_{\sS_n}}$ with
 face maps $(d_n)_{*}$ and degeneracy maps $(s_n)_{*}$. Also, let $H_{\bullet}$ be the simplicial ring defined by $[n] \mapsto \grot{\mot_{\sS_n}}$ with
 face maps $(d_n)^{*}$ and degeneracy maps $(s_n)^{*}$. We call $G_{\bullet}$ the {\it  simplicial grothendieck ring} of $\mot_{\sS}$ and $H_{\bullet}$ the {\it cosimplicial grothendieck ring} of $\mot_{\sS}$.

\begin{remark}
 We note that all of these results generalize to the case of $\infty$-simplicial sets -- that is, functors from 
 $\prod_{i=1}^{\infty} \Delta^o \to \categ{Sets}$. 
\end{remark}

\section{Specialization to schemic motivic measures}

Let $\sS$ be a simplicial sieve. We have already noted that $\ssieves{\sS} \subset \smes{\sS}.$ Moreover, 
for any fat point $\maxim$, we have
the function 
\begin{equation}
(\sX,\maxim) \mapsto \mu_{\maxim, \{\maxim\}, 0}^{\sim}(\sX{}) = [\sarc_{\maxim}\sX{}]
\end{equation}
which after post-composing with $\tau_0$ yields the motivic measure defined in \cite{schmot2}. Also, 
the function which sends a triple $(\sX{}, X, \maxim)$ to 
\begin{equation}
(\sX{}, X, \maxim) \mapsto \mu_{\spec \fld, \{\spec\fld\},0}^{\sim_1}(\sX{}) \cdot \mu_{\maxim, \{\maxim\}, 1}^{\sim_2}(X)
\end{equation}
is the function denoted by $\sX{} \mapsto \int \sX{} d_{\maxim}X$ in loc. cit. after post-composing with $\tau_0$. 
This covers all finite motivic measures studied in loc. cit.

In \cite{me}, the notion of a $\fat$-stable scheme $X$ where $\fat$ is an infinite limit point with point system $\bX$ was defined and studied. 
It is clear that $X$ is measurable and that $\tau_0(\mu_{\fat, \bX, 1}(\sX{A}))$ is the motivic 
measure denoted in \cite{me} by $\mu_{\fat}(A)$ where
$A$ is $\fat$-stable subset of $\arc{\fat} X$ where $\sX{A}$ is the limit sieve defined by $\sX{\maxim} = \pi_{\maxim}^{\fat}((A)_{\bullet})$
 for each $\maxim \in \bX$ where $\pi_{\maxim}^{\fat}$ is the natural map from $\sarc_{\fat} X $ to $\sarc_{\maxim}X$.
 Following my work in loc. cit., we define the notion of a laxly $\sS$-measurable limit simplicial sieve as follows.
 
 \begin{definition} \label{defrelmes}
Let $\sS$ be a simplicial sieve.
We say that a limit simplicial $\sS$-sieve $\sX{\star}$ is {\it laxly} $\sS$-{\it measurable} if there is a non-negative real number $Q$ such that
\begin{equation}\label{defrelmes}
\mu_{\fat, \bX, Q, l}^{\sim}(\sX{\star}) :=  \ulim {\class{\sX \maxim}\lef^{-\lceil Q\cdot\dim {\sarc_{\maxim}{\sX \ }}\rceil - l(\maxim)}}
\end{equation}
is an element of the image of the diagonal ring homomorphism
\begin{equation}
\grot{\ssieves {\sS}}_{\lef_{\bullet}} \to \prod_{\sim} \grot{\ssieves {\sS}}_{\lef_{\bullet}}
\end{equation}
where $\sim$ is the
ultrafilter on $\bX$ which defines the term on the right-hand side of Equation \ref{defrelmes} as an ultra-limit and $\sX{}$ is the base
of $\sX{\star}$ and $l$ is some set map from $\bX$ to $\nat$. Here
$\lef^{-\lceil Q\cdot\dim {\sarc_{\maxim}{\sX \ }}\rceil}$ is defined to be the multiplicative inverse of the equivalence class of the simplicial scheme defined by
$[n]$ to $\mathbb{A}_{S}^{-\lceil Q\cdot\dim {\narc_{\maxim}{\sX \ }}\rceil}$. 
We  denote the full subcategory of $\limssieves \sS$ whose objects are all laxly measurable limit simplicial $\sS$-sieves as 
$\categ{slax}_{\sS}$.
If we choose to fix the function $l$, we can form the subcategory $\smes{\sS}^{l}$ of all objects 
of $\categ{slax}_{\sS}.$ Thus, $\categ{slax}_{\sS} = \cup_{l} \smes{\sS}^{l}$. 
\end{definition}

\begin{remark}
 It is straightfoward to show that $\categ{slax}_{\sS}$ is a limit simplicial partial $\sS$-motivic sites
 for any simplicial sieve $\sS$. 
\end{remark}

 Again, following \cite{me}, if $X$ is a $\fat$-laxly stable scheme, then
 $\tau_0(\mu_{\fat, \bX, 1, l}^{\sim}(\sX{A}))$ is the motivic 
measure denoted in \cite{me} by $\mu_{\fat}^{l}(A)$ where
$A$ is $\fat$-laxly stable subset of $\arc{\fat} X$ where $\sX{A}$ is the limit sieve defined by $\sX{\maxim} = \pi_{\maxim}^{\fat}((A)_{\bullet})$.
This exhaust the infinite schemic motivic measures studied in loc. cit. 

\subsection{Indexed families}

Let $\categ{W}$ be a category and consider the simplicial category $\categ{sW}$ of $\categ{W}$ formed by all covariant functors from $\Delta^{\circ}\to \categ{W}$. Forgetting the face and degeneracy maps in $\categ{sW}$ gives us the category of all diagrams from a natural numbers to $\categ{W}$. In other words, it is the category formed by families of objects $w_i$ of $\categ{W}$ indexed by $i \in \nat$. We denote this category by $\categ{iW}$. Applying this forgetful functor to the simplicial categories within this paper yields a theory of {\it measurable indexed motivic sites relative to an indexed family of limit sieves}. Most of the results in this paper do not utilize the face and degeneracy maps. Therefore, by abstract non-sense,  the definitions and results carry over to this context without issue. 

\section{Toplogical realization}

To each object $A$ of $\sset$, we have the associated topolgical space 
\begin{equation}
|A| = \lim_{\Delta^n \to A}|\Delta^n|
\end{equation}
where $|\Delta^n|$ is the topological $n$-simplex. Moreover, this defines a functor $|-|$ from $\sset$ to $\categ{CGHaus}$
where $\categ{CGHaus}$ is the category of compactly generated hausdorff topological spaces.
We may extend this to a functor $\ssieves{\fld}$ to $\mbox{Func}(\fatpoints\fld,\categ{GCHaus})$ by defining 
$|\sX{}|(\maxim)= |\sX{}(\maxim)|$. 
This functor preserves unions, intersections products -- i.e. we have
\begin{equation}
\begin{split}
|\sX{}\cup\sY{}|(\maxim) &= |\sX{}|(\maxim)\cup|\sY{}|(\maxim) \\
|\sX{}\cap\sY{}|(\maxim) &= |\sX{}|(\maxim)\cap|\sY{}|(\maxim) \\
|\sX{}\times\sY{}|(\maxim) &= |\sX{}|(\maxim)\times|\sY{}|(\maxim) \ .
\end{split}
\end{equation}
Thus, given a motivic site $\mot$, we can form the associated {\it topological motivic site}, denoted by $|\mot|$, as the small subcategory of $\categ{CGHaus}$ whose objects are $|\sX{}|$ for $\sX{}\in \ssieves\fld$ whose morphisms are all morphisms in $\categ{CGHaus}$ of the form $|f| :|\sX{}| \to |\sY{}|$ where $f : \sX{}\to\sY{}$ is a morphism in $\ssieves\fld$. Note that $|-|$ sends $\mot$-homeomorphisms to (topological) homeomorphisms which are in $|\mot|$.
Thus, we can form the grothendieck ring of the topological motivic site $\grot{|\mot|}$ in the analogous way to the simplicial sieve case.

Of related interest, we have notion of homotopic maps in $|\mot|$: we say that two maps $f$ and $g$ are $\mot$-homotopic if there exists
a homotopy $H : |\sX{}| \times [0,1] \to |\sY{}|$ such that $H(-,t)$ is a morphism in $|\mot|$ for all $t$. Then, if we denote the equivalence class of $\sX{}\in \mot$ up to homotopy by $\langle \sX{} \rangle^h$, we can form the free abelian group generated by symbols $\langle \sX{} \rangle^h$ with $\sX{}\in \mot$. Here the equivalence relation is $\sX{} \sim \sY{}$ if and only if  $|\sX{}|$ is $\mot$-homotopic to  $|\sY{}|$. Moding out by scissor relations we form the homotopy grothendieck ring $\hgrot{\mot}$  of  $\mot$. We form the cooresponding topological homotopy grothendieck ring $\hgrot{|\mot|}$ in the analogous way.
Furthermore, we say that a limit $\fld$-sieve $\sX{\star}$ is measurable up to $\mot$-homotopy if 
if there is a non-negative real number $Q$ such that
\begin{equation}\label{defmes}
\mu_{\fat, \bX, Q}^{\sim}(\sX{\star}) :=  \ulim {\class{\sX \maxim}\lef^{-\lceil Q\cdot\dim {\sarc_{\maxim}{\sX \ }}\rceil}}
\end{equation}
is an element of the image of the diagonal ring homomorphism
\begin{equation}
\hgrot{\ssieves {\fld}}_{\lef_{\bullet}} \to \prod_{\sim} \hgrot{\ssieves {\fld}}_{\lef_{\bullet}}
\end{equation}
where $\sim$ is some 
ultrafilter.

Given any motivic site $\mot$, there is a surjective ring homomorphism 
\begin{equation}
\hgrot{\mot}\to \grot{\mot}
\end{equation}
 Thus, if $\sX{}$ is measurable up to $\mot$-homotopy, then it is measurable.

\end{document}